\documentclass{article}
\usepackage[utf8]{inputenc}

\title{Colored unavoidable patterns and balanceable graphs}
\author{Matt Bowen\thanks{McGill University, Canada. \texttt{matthew.bowen2@mail.mcgill.ca} } \and Adriana Hansberg \thanks{Instituto de Matem\'aticas, UNAM Juriquilla, Quer\'etaro, M\'exico. \texttt{ahansberg@im.unam.mx}} \and Amanda Montejano \thanks{UMDI, Facultad de Ciencias, UNAM Juriquilla, Quer\'etaro, M\'exico. \texttt{amandamontejano@ciencias.unam.mx}}\and Alp Müyesser\thanks{Freie Universit\"at Berlin and Berlin Mathematical School, Germany. \texttt{alp.muyesser@fu-berlin.de} }}
\date{\vspace{-5ex}}
\usepackage{amsmath}
\usepackage{amssymb}
\usepackage{amsthm}
\usepackage{changepage}
\usepackage{enumerate}
\usepackage{ dsfont }
\usepackage{soul}
\usepackage{hyperref}
\usepackage{graphicx}
\usepackage{xcolor}
\usepackage[normalem]{ulem}
\usepackage[a4paper, total={6in, 8in}]{geometry}

\newcommand{\ep}{\varepsilon}
\newcommand{\ex}{{\rm ex}}

\newcommand{\bal}{{\rm bal}}

\theoremstyle{plain}
\newtheorem{theorem}{Theorem}[section]
\newtheorem{lemma}[theorem]{Lemma}
\newtheorem{proposition}[theorem]{Proposition}
\newtheorem{claim}{Claim}[section]
\newtheorem{conjecture}[theorem]{Conjecture}

\newtheorem{definition}[theorem]{Definition}

\begin{document}
\maketitle

\begin{abstract}
  \par We study a Tur\'an-type problem on edge-colored complete graphs. We show that for any $r$ and $t$, any sufficiently large $r$-edge-colored complete graph on $n$ vertices with $\Omega(n^{2-1/tr^r})$ edges in each color contains a member from certain finite family $\mathcal{F}_t^r$ of $r$-edge-colored complete graphs. We conjecture that $\Omega(n^{2-1/t})$ edges in each color are sufficient to find a member from ${\mathcal{F}}_t^r$. A result of Gir\~ao and Narayanan confirms this conjecture when $r=2$. 
  \par Next, we study a related problem where the corresponding Tur\'an threshold is linear. We call an edge-coloring of a path $P_{rk}$ balanced if each color appears $k$ times in the coloring. We show that any $3$-edge-coloring of a large complete graph with $kn+o(n)$ edges in each color contains a balanced $P_{3k}$. This is tight up to a constant factor of $2$. For more colors, the problem becomes surprisingly more delicate. Already for $r=7$, we show that even $n^{2-o(1)}$ edges from each color does not guarantee existence of a balanced $P_{7k}$.
\end{abstract}

\section{Introduction}
\par The basic principle behind Ramsey theory is that no matter how a system is partitioned, there must exist an organized subsystem that is entirely contained inside one of the parts of the partition. Recently, numerous authors have pursued a line of research investigating the emergence of subsystems that are organized, yet meet every single part in the partition. Of course, to do so, one needs to assume that each part in the partition is sufficiently large. In order to state a prototypical result in this direction, we first give the following definition (to see results with a similar flavor which will not be addressed in the rest of this paper, see \cite{ADGMS, DMM,DiMu,LMT, MuTa}). Let $\mathcal{F}_t$ be the family of two-edge-colored complete graphs on $2t$ vertices where one color forms a clique of size $t$, or two disjoint cliques of size $t$. The following was conjectured by Bollob\'as, and proved by Cutler and Mont\'agh. 

\begin{theorem}[\cite{cutler}]\label{thm:cutler}
  \par Let $0 < \varepsilon \le \frac{1}{2}$ be a real number and $t \geq 1$ an integer. For large enough $n$, any two-edge-coloring of $K_n$ with at least $\ep\binom{n}{2}$ edges in each color contains a member of $\mathcal{F}_t$.
\end{theorem}

 \par 
 \par Cutler and Mont\'agh's argument shows that one can take $n\geq 4^{t/\varepsilon}$ to find a member of $\mathcal{F}_t$. Fox and Sudakov improved their result by showing that one can take $n\geq \varepsilon^{-ct}$ for some absolute constant $c$, which is tight up to the value of $c$ \cite{stuff}. In an attempt to generalize these results to a setting where an arbitrary number of colour classes are allowed, the following definition was given in \cite{multicolorbollobas}.
\begin{definition}[\cite{multicolorbollobas}]
  Let $k, t$ and $r$ be positive integers. Let $H : = K_{kt}$ be a complete graph on $kt$ vertices whose edges are $r$-colored. $H$ belongs to the family $\mathcal{F}_t^r$ if there is a partition $V(H) = \bigsqcup_{i\in[k]}V_i$ of the vertices of $H$ into $k$ parts, with $|V_i|=t$, such that:
  \begin{enumerate}
     \item For all $i,j\in[k]$, $H[V_i]$ and $H[V_i\times V_j]$ are monochromatic.
      \item All $r$-colors are present in $H$.
      \item Not all $r$ colors are present in $H\setminus V_i$, for any $i\in[k]$.
  \end{enumerate}
\end{definition}
\par We make a couple of remarks regarding this definition.  By $(1)$, the color of any edge in $H$ depends only on the parts the endpoints come from. By $(3)$, it is not hard to see that any graph in $\mathcal{F}_t^r$ can have at most $2r$ parts (see the argument in \cite{multicolorbollobas}). Hence, $\mathcal{F}_t^r$ is a finite set. Observe also that $\mathcal{F}^2_{t} = \mathcal{F}_t$. 
\par By a \textit{pattern}, we denote a maximal subfamily of $ \mathcal{F}_t^r$ consisting of graphs colored all the same up to permutations of their colors. There are two different patterns in $\mathcal{F}^2_t$ and nine different patterns in $\mathcal{F}^3_{t}$,  see an illustration of the latter in Figure \ref{fig:F3}. In general, the number of different patterns in $\mathcal{F}_t^r$ grows exponentially with $r$ as one can embed the family of all non-isomorphic tournaments on $r$ vertices into $\mathcal{F}_t^r$ by associating to each vertex a $K_t$ of a different color, and giving the bipartite graphs between the $K_t$'s the color of the $K_t$ they point to in the tournament (see \cite{Moon} for information regarding the number of non-isomorphic tournaments on $n$ vertices).
\par We can now state the multicolor generalization of the result of Cutler and Mont\'agh, due to Lamaison, and the first and the last author. 
\begin{theorem}[\cite{multicolorbollobas}]\label{thm:multicolorbollobas}
For any $r\geq 2$, there exists a constant $c:=c(r)$ such that, for any $\ep>0$ and $t\geq 2$, any $r$-coloring of a $K_n$ with $n\geq \ep^{-ct}$ and $\ep\binom{n}{2}$ edges in each color contains a member of $\mathcal{F}_t^r$.
\end{theorem}

\subsection{Tur\'an-bounds for colored unavoidable patterns}
The bound in Theorem \ref{thm:multicolorbollobas}, up to the dependence of $c$ on $r$, is optimal by a random construction, similar to the one given in \cite{stuff}. However, if one is interested in the minimum density of edges required of each color in order to force a member of $\mathcal{F}_t^r$, assuming that each color has $\Theta(n^2)$ many edges is not necessary. To discuss this extremal aspect of the problem precisely, we first define the following Tur\'an-type parameter. 
\begin{definition}
Let $r, t, n$ be positive integers. Let $\mathcal{F}$ be a family of $r$-edge colored graphs. We denote by $\ex_r(K_n, \mathcal{F})$ the minimum integer $m$ (if it exists) such that, for any $r$-edge coloring of $K_n$ with more than $m$ edges in each of the $r$ colors, $K_n$ contains a member of $\mathcal{F}$. If there is no such $m$, we set $\ex_r(K_n, \mathcal{F})=\infty$. 
\end{definition} 

We begin by a discussion of the the case when $r=2$. Caro and the second and third author showed that there exists a $\delta:=\delta(t)$ such that $\ex_2(K_n, \mathcal{F}_t) = \Omega(n^{2-\delta})$ when $n$ is large enough \cite{caro}. Shortly after, Gir\~ao and Narayanan \cite{girao} proved that $\delta=1/t$ is best possible up to the involved constants, supposing that the well-known conjecture that $\ex(K_{t,t})=\Omega(n^{2-1/t})$ for all $t$ is true \cite{KoSoTu} (we give a shorter proof of their result in Section~\ref{sec:shorterproof}).

\begin{figure}\label{fig:F3}
\begin{center}
\includegraphics[]{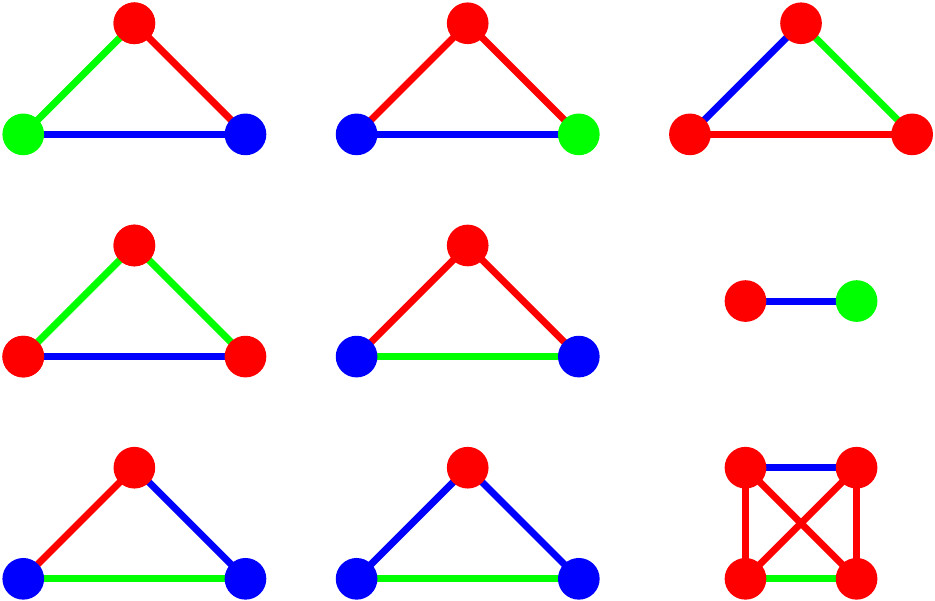}
\end{center}
\caption{The patterns from $\mathcal{F}_{3,t}$. The colored vertices represent cliques of the corresponding color, of size $t$. Similarly, edges represent complete bipartite graphs of the corresponding color between the associated cliques.}
\end{figure}

For arbitrary $r \ge 3$, the structure of the graphs in $\mathcal{F}_t^r$ is a lot more complicated (already for $r=3$ there are $9$ different patterns, see Figure \ref{fig:F3}). However, it is still natural to suspect that finding large bipartite graphs are the only barrier for finding the patterns in $\mathcal{F}_t^r$, meaning that $\Omega(n^{2-1/t})$ edges in each color class should guarantee the existence of a member from  $\mathcal{F}_t^r$. We conjecture that this is indeed the case.  

\begin{conjecture}\label{mainconjecture}
For any $r\geq 2$ and $t\geq 1$, there exists a constant $C:=C(r,t)$ such that,  for $n$ large enough,  $\ex_r(K_n, \mathcal{F}_t^r) \le Cn^{2-1/t}$.
\end{conjecture}
We show that at this same density (with $\Omega(n^{2-1/t})$ edges) we can find a member from $\mathcal{F}_t^{s}$, where $s =  \left\lfloor\frac{t}{r^r}\right\rfloor$. 
\begin{theorem}\label{thm:extremalmulticolorbollobas}
For any $r\geq 2$ and $t\geq 1$, there exists a constant $C:=C(r,t)$ such that,  for $n$ large enough, $\ex_r(K_n, \mathcal{F}_t^r) \le Cn^{2-1/ t'}$,  where $t' =tr^{r-1}$.
\end{theorem}

Observe that the result of Gir\~ao and Narayanan \cite{girao} does not follow from the above by setting $r=2$. We discuss why the existing techniques fail to give optimal bounds for the above problem, how one could potentially remove the $r^{r-1}$ factor, and why the case of $r=2$ is much easier, in Section \ref{sec:generalbounds}. In particular, we give a short proof of the result of Gir\~ao and Narayanan in Section~\ref{sec:shorterproof}.

\subsection{Balanceable graphs}
Next, we turn our attention to a more specific problem, concerning balanced colorings of paths.  
We say that an $r$-colored graph $G$ is \emph{balanced} if each of the $r$ colors appears in precisely $\lfloor e(G)/r \rfloor$ or $\lceil e(G)/r \rceil$ edges. Further, we say that an $r$-coloring of $E(K_n)$ contains a \emph{balanced copy} of $G$ if it admits a balanced embedding of $G$. The following parameter for the case $r=2$ was introduced by Caro and the second and third author \cite{caro}.

\begin{definition}\label{def:bal_r} Given a graph $G$ and positive integers $r$ and $n$,  we set $\bal_r(n, G)=\ex_r(K_n, \mathcal{F}_{\rm bal}(G))$,  where $\mathcal{F}_{\rm bal}(G)$ is the family of all $r$-colored copies of $G$ in which each of the $r$ colors appears in either $\lfloor e(G)/r \rfloor$ or $\lceil e(G)/r \rceil$ edges.  We call $\bal_r(n, G)$ the \emph{$r$-balancing number} of $G$ and, when $r=2$, we will put $\bal_2(n, G) = \bal(n,G)$ and call it just the \emph{balancing number} of $G$.  A graph $G$ with $\bal_r(n, G) < \infty$ for every sufficiently large $n$ is called \emph{$r$-balanceable} or, when $r=2$, simply \emph{balanceable}.
\end{definition}

Caro et al. \cite{caro} characterized all balanceable graphs. It follows by this characterization that all balanceable graphs $G$ have $\bal(n, G)=o(n^2)$.  There are many non-balanceable graphs as well as many balanceable graphs, see \cite{caro, DEHV, DHV} for several examples.  There are many dense balanceable graphs (like some amoebas \cite{CHM3, caro}) and also many graphs having linear balancing number (in $n$), like trees \cite{caro},  and cycles $C_n$ with $n \not\equiv 2 \;({\rm mod}\; 4)$  \cite{DEHV}.  For instance, $\bal(n,P_{2k}) = \left(\left\lfloor\frac{k-1}{2} \right\rfloor+ o(1) \right) n$, and the value was precisely determined in \cite{caro}. Surprisingly, it turns out to be an intricate problem to determine even the order of magnitude of $\bal_r(n, P_{rk})$ for arbitrary $r$ and $k$, as the next two results show.

\begin{theorem}\label{negative}
Let $k$ be an odd natural. Then there exist infinitely many $r$ such that $\bal_r(n,P_{rk})=\Omega(n^2)$. In particular, $\bal_7(n,P_{7k})=\Omega(n^2)$.
\end{theorem}
We would expect that, in this scenario, $\bal_r(n,P_{rk})$ is in fact $\infty$, but this would require giving a construction with exactly the same number of edges from each color class. On the other hand, we have the following positive result:
\begin{theorem}\label{positive}
Let $r\geq 2$. Then, there exists a $k_0$ such that for all $k\geq k_0$, $\bal_r(n, P_{2rk})=o(n^2)$.
\end{theorem}
We remark that while the dependence of $k_0$ on $r$ our proof gives is possibly far from optimal, some threshold is necessary. Indeed, there exist members of $ \mathcal{F}_t^r$ for every $r$ where one color class separates all the remaining colors, thus making it impossible to embed a balanced $P_{r}$ (a rainbow path) in such graphs. As an example of such a member from  $\mathcal{F}_t^r$,  one may start with $r-1$ vertex-disjoint $K_{n,n}$, each colored a distinct color, and color all the edges that remain in the $r^{th}$ color. An illustration can be found for $r=3$ in the bottom-right part of Figure \ref{fig:F3}.

\par Having established that degeneracies arise for large $r$, we turn our attention to the function $\bal_3(n, P_{3k})$. The following theorem establishes the growth of this function up to a constant factor of $2$. 

\begin{theorem}\label{thm:generalbounds}
For $k \ge 1$, $(\frac{1}{2}(k-1)+ o(1)) n \le \bal_3(n,P_{3k}) \le (k+ o(1))n $.
\end{theorem}

We believe that the lower bound from Theorem \ref{thm:generalbounds} should be tight for arbitrary $k$ but are only able to confirm this when $k \le 2$. When $k=1$, it is not hard to show $\bal_3(n,P_{3})=0$ (Proposition~\ref{prop:p3is0}), whereas for $k =2$ some more work is needed. We state the latter result in the following theorem.

\begin{theorem}\label{thm:p6}
  $\bal_3(n, P_{6})=(\frac{1}{2}+o(1))n$.
\end{theorem}

\subsection{Organization of the paper}
We begin by proving our general result, Theorem \ref{thm:extremalmulticolorbollobas}. This already implies $\bal_3(P_{3k})=o(n^2)$, as one can check all members from $\mathcal{F}_t^3$ in Figure \ref{fig:F3} admit balanced embeddings of $P_{3k}$. Afterwards, we prove much better bounds on this function, showing that in fact it grows linearly in $n$ (Theorems \ref{thm:generalbounds} and \ref{thm:p6}). We proceed by proving Theorems \ref{negative} and \ref{positive}, which summarize our understanding of the function $\bal_r(n, P_{rk})$ for arbitrary $r$. We conclude with some open problems in the Discussion section. 
\section{Proof of Theorem \ref{thm:extremalmulticolorbollobas}}\label{sec:generalbounds}
We will use the following version of the dependent random choice lemma.

\begin{lemma}[\cite{DRC}]\label{lemma:drc}
For all $K,t\in\mathbb{N}$, there exists a constant $C$ such that any graph with at least $Cn^{2-1/t}$ edges contains a set $S$ of $K$ vertices in which each subset $X\subseteq S$ with $t$ vertices has a common neighborhood of size at least $K$. 
\end{lemma}

\par We will also use the following, which is essentially a consequence of iterating a bipartite version of the dependent random choice lemma from \cite{multicolorbollobas}.

\begin{lemma}[Corollary 2.6 in \cite{multicolorbollobas}]\label{lemma:superdrc}
Let $r$ and $t$ be positive integers, $r\geq 2$. There exists $N=N(r,t)$ such that the following holds for all $n\geq N$. Let $A_1, A_2, \dots, A_t$ partition the vertex set of an $r$-colored complete graph, where $|A_i|=n$ for all $i\in[t]$. Then, there exist subsets $X_i\subset A_i$, of size $|X_i|=\frac{1}{2^{t+1}r}\log_rn$, such that every set $X_i$ is monochromatic and every complete bipartite graph between $X_i$ and $X_j$ is monochromatic.
\end{lemma}
We emphasize that the proof structure of the result in this section will be very similar to that of the main result from \cite{multicolorbollobas}. We include the details for completeness, and to set up the scene for a discussion of the difficulty of removing the $r^r$ term in the exponent. 

\begin{proof}[Proof of Theorem \ref{thm:extremalmulticolorbollobas}]
Start with given integers $r\geq 2$, $t\geq 1$ and set $t'=tr^{r-1}$.  We want to show that there exists a constant $C:=C(r,t)$ such that, for $n$ large enough,  $\ex_r(K_n, \mathcal{F}_t^r) \le Cn^{2-1/t'}$.
We choose a $C$ with the benefit of hindsight, large enough to make sure the following calculations go through. Similarly, we choose a sufficiently large $n$, and consider an $r$-coloring of a $K_n$ with $Cn^{2-1/t'}$ edges in each color class. 
\par First, we choose a $K'$ (with hindsight) such that $K'\gg t'$, and apply,for each color class, Lemma~\ref{lemma:drc} with $t=t'$ and $K=rK'$ to find a collection of  $rK'$-sized sets $\{S_i\}^r_{i=1}$, satisfying that all its $t'$-subsets have $rK'$ common neighbors in color $i$ (this can be done if $C$ is large enough). The different sets $S_1, S_2,\cdots,S_r$ can intersect,  but we can manage to choose,  without renaming, suitable subsets of each one in order to have $r$ disjoint sets each with $|S_i|= K'$.
\par We now apply Lemma \ref{lemma:superdrc} to the collection $\{S_i\}^r_{i=1}$, obtaining sets $\{X_i\}^r_{i=1}$ such that, for each $i\in[r]$, $X_i\subset S_i$, $|X_i|=\frac{1}{2^{t'+1}r}\log_r(K')$,   $X_i$ induces a monochromatic clique (of some color, not necessarily $i$), and the complete bipartite graph connecting vertices between  $X_i$ and $X_{j}$, with $i\neq j$, is  monochromatic. Here, we make sure to select $K'$  large enough to ensure that  $|X_i|$ exceeds $t'$.
\par Now, we use the property of the sets $S_i$ to find $rK'$ common neighbors of $X_i$ in color $i$. Call these sets  $Y_i$ and, again, choose suitable subsets (without renaming) to ensure that we get a collection $\{Y_i\}^r_{i=1}$ of disjoint sets, where $|Y_i|=K'$ for each $i\in[r]$.
\par Our aim is now to find monochromatic complete bipartite graphs between $Y_i$ and $X_{j}$ for all  $i\neq j$. 
Associate to each vertex $y$ of $Y_i$ a $t'r$-long tuple with entries in $[r]$ encoding the color of the edges $(y, x)$ where $x\in \bigcup X_i$ (recall  $|\bigcup X_i|=t'r$). As there are at most $r^{t'r}$ such tuples, there must be $|Y_i|/r^{t'r}$ vertices in each $Y_i$ (call them $Y'_i$) such that the associated tuple is identical. This means that for each $x\in X_{j}$ all the edges  form $x$ to vertices in $Y'_i$, are of the same color.

\par We may now fix subsets $X'_i\subseteq X_i$ of size $|X_i|/r^{r-1}$ so that the edges between $Y'_{j}$ and $X'_{i}$, for $i\neq j$, are monochromatic (recall that the graphs are already monochromatic when $i=j$). Indeed, associate to each vertex $x$ of $X_i$ an $(r-1)$-long tuple with entries in $[r]$ encoding the color of the edges from $x$ to the sets $Y'_{j}$, $j\neq i$. So, for at least $|X_i|/r^{r-1}$ of the vertices its tuples are all identical.  Call these sets $X'_i$ and note that $|X'_i|\geq  t'/r^{r-1} = t$ for every $i\in[r]$.
\par To finish, we apply Lemma \ref{lemma:superdrc} to the collection $\{Y'_i\}^r_{i=1}$, obtaining sets $\{Y''_i\}^r_{i=1}$ such that, for each $i\in[r]$, $Y''_i\subset Y'_i$, $|Y''_i|=\frac{1}{2^{t'+1}r}\log_r(K'/r^{t'r})$,   $Y''_c$ induces a monochromatic clique and the complete bipartite graph connecting  vertices between  $Y''_i$ and $Y''_{j}$, for $i\neq j$, is  monochromatic. We select $K'$ to make sure this quantity exceeds $t$. Looking at the graph induced by the sets  $\{X'_i\}^r_{i=1}$ and  $\{Y''_i\}^r_{i=1}$, we  thus have a blow-up of a complete graph of order $2r$ with each blow-up of size at least $t$, where all $r$ colors are used, so this structure must contain as a subgraph a member from  $\mathcal{F}_t^r$.  Hence, $\ex_r(K_n, \mathcal{F}_t^r) \le Cn^{2-1/t'}$.
\end{proof}

\subsection{Removing the $r^r-1$ factor}\label{sec:improvements}
\par We now discuss the deficiency of the above proof, namely the loss of a constant factor of $r^r-1$, which we conjecture to be unnecessary.  We lost this factor while trying to find monochromatic bipartite subgraphs between $X_i$ and $Y'_j$, for $i \ne j$.  Obviously, to do so, one needs to shrink $X_{i}$ by some factor.  Observe that, for the case $r=2$,  this step is not needed (as can be seen in the proof of Theorem \ref{thm:GiNa} given in the next subsection).  For $r \ge 3$,  the issue is that at this stage of the proof, we have already applied dependent random choice (Lemma \ref{lemma:drc}), and fixed $|X_{i}|$ to be a subset of size $t$. Once we do this, we cannot afford any further shrinking of $X_{i}$ to obtain the conjectural bound. One could hope to find the monochromatic bipartite graphs in a different order, perhaps finding all the bipartite graphs that are incident on $X_1$, and proceeding inductively. The issue is once two subsets (say $A$ and $B$) are fixed that are the two joint monochromatic neighborhoods of $X_1$ between which we are supposed to find a bipartite graph of a different color (say $2$), we have no guarantee of finding even a single $2$-colored edge between $A$ and $B$. If $A$ and $B$ were linear sized subsets of a subgraph which is \textit{regular} in color $2$ however, this wouldn't be an issue. Here, regular is in the sense of the Sparse Regularity Lemma (for example, see \cite{scott}). Of course, for the regularity assumption to be useful, one needs more delicate dependent random choice type lemmas, which guarantee linear sized neighborhoods. This turns out to be not a serious issue, as there are such lemmas, for example see Lemma 6.3 in \cite{DRC}. 
\par There is a caveat here, namely that if one wishes to find linear sized neighborhoods via a dependent random choice type lemma, one needs to tolerate a small fraction of subsets which don't have the desired large neighborhoods. However, we would only need a single subset with a large neighborhood, such that the subset is a clique. If the graph is large enough, there will be enough subsets that are monochromatic cliques so that one of them will have a large neighborhood. 
\par So we believe that such a regularity based approach might prove to be fruitful to remove the $r^{r-1}$ factor, but it is tricky to find even a single regular pair with many edges in sparse graphs. As noted in \cite{scott}, when one applies the Sparse Regularity Lemma to a graph, it could be that none of the edges end up being between regular pairs of the given partition. This is in stark contrast to the dense case, where finding a single regular pair is a trivial consequence of the Szemerédi regularity lemma.

\subsection{The case of $r=2$}\label{sec:shorterproof}
In this subsection, in an attempt to demonstrate the relative simplicity of Conjecture~\ref{mainconjecture} when $r=2$, we give a very short proof of this particular case, originally proved by Gir\~ao and Narayanan \cite{girao} using a more involved approach. Our version of the proof will essentially be an adaptation of an argument from \cite{caro}, replacing the usage of the K\H{o}vari-S\'os-Tur\'an theorem with the dependent random choice lemma (Lemma \ref{lemma:drc}). 
\par Before we begin, we recall two definitions. Denote by $R(k)$ the classical Ramsey number, that is, the smallest integer for which every $2$-edge-coloring of $K_n$, with $n\geq R(k)$, contains a monochromatic $K_k$. Denote by $BR(k)$ the bipartite Ramsey number, that is, the smallest integer for which every $2$-edge-coloring of $K_{n,n}$, with $n\geq BR(k)$, contains a monochromatic $K_{k,k}$. All we need in the following proof is that $R(k)$ and $BR(k)$ are constants depending only on $k$, which is a well-known fact.

\begin{theorem}[\cite{girao}]\label{thm:GiNa}
For any $t\geq 1$, there exists a constant $C:=C(t)$ such that,  for $n$ large enough,  $\ex_2(K_n, \mathcal{F}_t) \le Cn^{2-1/t}$.
\end{theorem}

\begin{proof}
Given $t$, consider a $2$-edge-coloring of $K_n$ (for $n$ large) with $Cn^{2-1/t}$ edges in each color class, where $C$ is the constant given in Lemma \ref{lemma:drc} for $t$ and $K=R(BR(t))$. Then we can find a red $K_{t,BR(t)}$ where both parts of the bipartite graph are monochromatic cliques. If either of these cliques are blue, we find a member of $\mathcal{F}_t$. So we may assume that both are red and thus, there is a red clique of size $BR(t)$. Similarly, we may assume that there is a blue complete graph of the same order disjoint with the previous one. Consider now the $2$‐edge colored complete bipartite graph, $K_{BR(t),BR(t)}$, induced by the vertices of those two cliques. By definition, there is a monochromatic $K_{t,t}$ in such complete bipartite graph, yielding the desired graph contained in $\mathcal{F}_t$.
\end{proof}

\par We remark that the proof above is quite specific to the $r=2$ case. Indeed, if we employed a similar strategy for the $r=3$ case, for each monochromatic $K_{t,BR(t)}$ we find where both parts of the bipartite graph are monochromatic cliques, all we could conclude is that the entire structure does not use all three colors. In particular, instead of finding three large cliques of three different colors, we could end up with three large cliques with all the same colors.

\section{Proofs of Theorems \ref{thm:generalbounds} and \ref{thm:p6}}

Here we focus on upper and lower bounds for the function $\bal_3(n, P_{3k})$.  For convenience, we recall the corresponding result from the $2$-color case, this time in its most precise formulation.

\begin{theorem}[\cite{caro}]\label{thm:bal_2_paths}
Let $k \ge 1$ be even and let $n \ge \frac{9}{8}k^2 + \frac{1}{2}k + 1$. Then
\begin{equation*}
    \bal_2(n, P_{k}) = \begin{cases}
               \frac{(k-2)n}{4} -\frac{k^2}{32} +\frac{1}{8}               & k\equiv 2 \pmod{4} \\
               \frac{(k-4)n}{4} -\frac{k^2}{32} + \frac{k}{8} +1 & k\equiv 0 \pmod{4}
           \end{cases}
\end{equation*}

\end{theorem}

We remark that in \cite{caro} the extremal family of $2$-edge-colored $P_{2k}$-avoiding complete graphs were  explicitly characterized.
\subsection{Preliminaries}
In preparation for the proofs of Theorems \ref{thm:generalbounds} and \ref{thm:p6}, we first collect some helpful lemmas. In the proofs below, we always assume that the host graph has sufficiently many vertices. We start giving the $3$-balancing number for $P_3$. Observe that a $3$-balanced $P_3$ is also known as a rainbow $P_3$.

\begin{proposition}\label{prop:p3is0} 
$\bal_3(n, P_{3})=0$.
\end{proposition}
\begin{proof}
  Assume that we have a large $3$-edge-colored complete graph with at least $1$ edge from each color. There must exist a triangle with exactly two of the colors represented, say red and blue. No green edge can be adjacent to this triangle (otherwise we we can easily construct  a rainbow $P_3$), but there must be a green edge. So, casing on the color of the edge between the green edge and one of the vertices of the triangle adjacent to both red and blue, we see that a balanced $P_3$ arises either way.
\end{proof}

We will use notation $P=x_0x_1\dots x_k$ to represent a $k$-path with edges $x_ix_{i+1}$ for $i\in\{0,\dots,k-1\}$.  Also we will call a $3$-balanced $K_3$ a \emph{rainbow triangle} because this is the term that is used in connection to Gallai-colorings, which we will use later.

\begin{lemma}\label{lem:norainbow} Let $k\geq 2$. If a $3$-edge-colored complete graph contains a balanced $P_{3k-3}$ and a vertex-disjoint rainbow triangle, then the graph contains a balanced $P_{3k}$.
\end{lemma}
\begin{proof}
Let $P=x_0x_1\dots x_{3k-3}$ be  a balanced path which is vertex-disjoint  to a rainbow triangle $\{x_r,x_g,x_b\}$ where $x_gx_b$ is red, $x_rx_g$ is blue and $x_bx_r$ is green. If $x_0x_r$ or $x_{3k-3}x_r$ is blue or green, we we can easily construct  a balanced $P_{3k}$ (take $Px_rx$ or $xx_rP$ where $x=x_b$ or $x=x_g$ according to the color we choose). Thus, we may assume that both $x_0x_r$ and $x_{3k-3}x_r$ are  red. Similarly, avoiding balanced $P_{3k}$s,  the colors of $x_0x_b$, $x_{3k-3}x_b$, $x_0x_g$ and  $x_{3k-3}x_g$ are determined to be blue and green respectively. Suppose now, without loss of generality,  that the first edge of the path $P$, $x_0x_1$,  is red. Then the path $x_{1} \dots x_{3k-3}x_bx_gx_{0}x_r$ is a balanced $P_{3k}$.
\end{proof}

Now we show the lower bound from Theorem \ref{thm:generalbounds}. Let $k\geq 2$ and consider a $3$-edge-coloring of $K_n$ where we split $V(K_n)$ into three parts $A$, $B$, and $C$, with $|A|=k-1$ and $|B|=|C|=(n-k+1)/2$, and we color the edges as follows: edges between $A$ and $B$ red, edges between $A$ and $C$ blue, and all remaining edges in green. Since any balanced $P_{3k}$ requires $k$ red edges and $k$ blue edges, and any such edge is incident on $A$, this graph contains no balanced $P_{3k}$. Further, the graph has at least $(k-1)(n-k+1)/2$ edges in each color class, implying
\begin{equation}\label{eq:lb}
\bal_3(n, P_{3k})\geq \left(\frac{k-1}{2}+o(1)\right)n
\end{equation}
as desired.

Before proving the upper bound of Theorem \ref{thm:generalbounds}, we will first determine the $3$-balancing number for $P_6$. This will show that we are able to match the lower bound (\ref{eq:lb}) when $k=2$. 

\subsection{Proof of Theorem \ref{thm:p6}}

In order to prove Theorem~\ref{thm:p6} we need the following lemmas.

\begin{lemma}\label{lem:1} If a $3$-edge-colored complete graph contains two vertex-disjoint balanced $P_3$s whose middle edges are different colors, then the graph contains a balanced $P_6$.
\end{lemma}
\begin{proof} Without loss of generality, assume that the first path $P=x_0x_1x_2x_3$ is red-blue-green, and the second path $P'=x_4x_5x_6x_7$ is red-green-blue. If $x_3x_4$ is blue or red, we have a balanced $P_6$  (take $PP'$ and remove either $x_0$ or $x_7$ according to the color of $x_3x_4$). Similarly if $x_0 x_7$ is green or red, we have a balanced $P_6$. Assuming these edges are green and blue respectively, we can see that $x_7x_0x_1x_2x_3x_4x_5$ is a balanced $P_6$.
\end{proof}

The above lemma will allow us to conclude that all balanced $P_3$s in a $3$-edge-colored complete graph without balanced $P_6$s must be color isomorphic.  This fact combined with the next lemma will allow us to conclude that  in a $3$-edge-colored complete graph with several balanced $P_3$s and no balanced $P_6$,  the color of the middle edge of any of  the $P_3$s must be dense in the complete graph.

\begin{lemma}\label{lem:2} If a $3$-edge-colored complete graph without a balanced $P_6$ contains  two vertex-disjoint balanced $P_3$s whose middle edges are of the same color, say blue, then any other edge between the vertices of these two paths (including edges contained within a single path) is blue.

\end{lemma}
\begin{proof}
Let $P=x_0x_1x_2x_3$ and  $P'=x_4x_5x_6x_7$ be two vertex-disjoint  red-blue-green paths. If $x_3x_4$ is red or green, we we can easily construct  a balanced $P_{6}$ (take $PP'$ and remove either $x_0$ or $x_7$ according to the color of $x_3x_4$). Hence, $x_3x_4$ is blue. The same argument works to conclude that $x_0x_7$ is blue. Observe   that $C=x_0x_1\dots x_7x_0$ is a red-blue-green-blue-red-blue-green-blue cycle. If $x_0x_2$ is red then $x_0x_2x_3x_4x_5x_6x_7$ is a balanced $P_6$. Also, by  Lemma \ref{lem:norainbow}, $x_0x_2$ can't be green (see triangle $\{x_0,x_1,x_2\}$ and path $x_4x_5x_6x_7$). Thus, $x_0x_2$ must be blue. A similar argument works to prove that all edges $x_ix_{i+2}$, with $i\in\{0,\dots,7\}$ (where addition is taken modulo $8$) are blue. Note now that $x_3x_7$ can't be red (respectively, green) by $Px_7x_6x_5$ (respectively, $Px_7x_5x_4$). Thus, $x_3x_7$ must be blue. Again, by taking advantage of the symmetry of $C$, we can conclude that all edges of the form $x_ix_{i+4}$, with $i\in\{0,\dots,7\}$ and addition modulo $8$, are blue. Finally,  to conclude that the remaining edges are blue, we  consider a new cycle $C'=x_0x_1x_3x_2x_4x_5x_7x_6x_0$ which is color isomorphic to $C$ and repeat the arguments.
\end{proof}

We are now in a position to prove Theorem \ref{thm:p6}, in which we state that $\bal_3(n, P_{6})=(\frac{1}{2}+o(1))n$.

\begin{proof}[Proof of Theorem \ref{thm:p6}]
  Let $\ep>0$ be arbitrary, and consider a $3$-edge-coloring of  $K_n$ (for $n$ sufficiently large) with each color class having at least $(\frac{1}{2}+\ep)n$ edges. By Proposition \ref{prop:p3is0}, we can find a balanced $P_3$. Let $I$ be the vertex set of a maximum sized family of vertex-disjoint balanced $P_3$s and let $D = V(K_n) \setminus I$. Let $i=|I|$. (So $I$ contains $i/4$ vertex-disjoint balanced $P_3$s.) By Proposition \ref{prop:p3is0}, one of the color classes must not appear in $D$. We will handle the case where $i=1$ and $i\geq 2$ separately.\\
 
 \noindent
\textit{\textbf{Case 1: Assume $i=1$.}} \\
Let  $P$ be the unique balanced $3$-path in $K_n$. Assume, without loss of generality that red is missing in $D$. We first show that it cannot be the case that $D$ is almost missing another color as well.\\
 
  \noindent
 \textit{\textbf{Subcase 1.1: Suppose $D$ is monochromatic except for at most one edge.}}\\
Assume, without loss of generality, that all edges in $D$ but possibly one edge, say $uv$, are blue (so $uv$ is either blue or green). Observe that between $D \setminus\{u,v\}$ and $I$ we have almost all (but a constant) of the red and the green edges.  Since all color classes have at least $(\frac{1}{2}+\ep)n$ edges, it cannot be the case that all green and red edges are incident to the same vertex in $P$. Thus, there are two distinct vertices $x,y\in I$ and five distinct vertices (since $n$ is sufficiently large)  $x_1,x_2,y_1,y_2,z\in D \setminus\{u,v\}$ such that $xx_1$ and $xx_2$ are green, $yy_1$ and $yy_2$ are red, and so $x_1xx_2zy_1yy_2$ is a balanced $P_6$.\\

  \noindent
\textit{\textbf{Subcase 1.2: Suppose $D$ has at least two blue edges and at least two green edges.}} \\
By Theorem \ref{thm:bal_2_paths}, the $2$-color balancing number of $P_4$ is $2$, so we can find a  blue-green-balanced  $P_4$ in $D$, say $Q$.  Observe that between $V(P)$ and $D\setminus V(Q)$ we have almost all (but a constant) of the red edges.  Since all color classes have at least $(\frac{1}{2}+\ep)n$ edges,  there must be two distinct vertices $x,y\in D\setminus V(Q)$ and one vertex in $z\in V(P)$ such that $zx$ and $zy$ are red. Let $a$ and $b$ be the end-vertices of the path $Q$. Consider now the cycle $C=zxQyz$. We know that both edges $xa$ and $yb$ are either blue or green. Suppose first that $xa$ and $yb$ are both of the same color, say blue. Then the path $xQy$ has $4$ blue edges and $2$ green edges. Thus, it must contain two consecutive blue edges, say $rs$ and $st$. Then $C - s$ is a balanced  $3$-colored $6$-path. If, on the other side, $xa$ and $yb$  are one blue and one green, then we can take two consecutive edges $rs$ and $st$ from $Q$ such that they have different color and we can see that $C - s$ is a balanced $3$-colored $6$-path.\\
  
  \noindent
\textit{\textbf{Case 2: Assume $i\geq 2$.}} \\
By Lemma \ref{lem:1}, we may assume, without loss of generality, that all vertex disjoint balanced paths in $I$ are of the form red-blue-green. By Lemma \ref{lem:2}, all remaining edges induced by vertices in $I$ are blue. In this case we will conclude that either there are too few red or green edges, or we can find a balanced $P_6$.  \\

  \noindent
\textit{\textbf{Subcase 2.1: $D$ has no red or no green edges.}} \\
Assume, without loss of generality, that red is missing in $D$.  Since we have at least $(\frac{1}{2}+\ep)n$ red edges in $K_n$ and amongst them only $\frac{i}{4}$ in $I$ and none in $D$, the remaining red edges are all in $E(I,D)$. Let $uv$ be a red edge with $u \in I$ and $v \in D$. Then $u$ belongs to a red or a green edge $ux$ in $I$. Assume first that $ux$  is green. Consider another green edge $x_1x_2$  and a red edge $y_1y_2$ in $I$. Since all these edges form a matching and all other edges between the vertices $u, x_1, x_2, y_1, y_2$ are blue, we can see easily that $vuxx_1x_2y_1y_2$ is a balanced $P_6$. The case that $ux$ is red is completely analogous by taking two green edges $x_1x_2$ and $y_1y_2$.\\

\noindent
\textit{\textbf{Subcase 2.2: $D$ has at least a red edge and a green edge.}} \\
Then blue has to be missing in $D$. Let $uvw$ be a green-red path. Consider one red edge $x_1x_2$ and one green edge $y_1y_2$ in $I$. Let $z \in I \setminus \{x_1,x_2,y_1,y_2\}$. Then all other edges between the vertices $x_1,x_2,y_1,y_2$ and $z$ are blue. Casing upon the color of the edge $wx_1$, it is easily seen that there is a balanced $P_6$ within these vertices. For example, if $wx_1$ is red, then $uvwx_1y_1y_2z$ is a balanced $P_6$. The other two cases are done similarly.
\end{proof}

\subsection{Proof of upper bound of Theorem \ref{thm:generalbounds}}

 In general, the best upper bound for $\bal_3(n, P_{3k})$ we have is the following.
\begin{equation}\label{eq:ub}
\bal_3(n, P_{3k})\leq (k+o(1))n.
\end{equation} In order to prove (\ref{eq:ub}), we will use the following theorem. Recall that the extremal number $ex(n, H)$ is the maximum number of edges of a $H$-free graph on $n$ vertices. Erd\H{o}s and Gallai \cite{ErGa}  determined the extremal number for paths. From this result, it follows that

\begin{equation}\label{eq:ErGa}
ex(n, P_k) \leq \frac{k-1}{2}n.
\end{equation}

Now, we are ready to prove (\ref{eq:ub}). Given a coloring on the edges of $K_n$ with red ($r$), blue ($b$) and green ($g$), we will use the following notation. For $c \in \{r,b,g\}$ and a set $S \subseteq V(K_n)$, we will denote by $e_c(S)$ the number of $c$-colored edges with both vertices in the set $S$. If $S, T \subseteq V(K_n)$ are two disjoint sets, we set $E(S,T)$ for the set of edges with one vertex in $S$ and one in $T$, and $E_c(S,T)$ for the set of $c$-colored edges from $E(S,T)$. Moreover $|E(S,T)| = e(S,T)$, $|E_c(S,T)| = e_c(S,T)$, and $E_c(v,S) = E_c(\{v\},S)$ for a vertex $v$.

\begin{proof}[Proof of  (\ref{eq:ub})] Let $0 < \epsilon < \frac{1}{2}$ be arbitrary and let $n$ be large enough such that all inequalities hold. Consider a $3$-edge-coloring of $K_n$ with at least $(k+\epsilon)n$ edges of each color class. We proceed by induction on $k$ to prove that there is a balanced $P_{3k}$. The case $k=2$ was already done in Theorem \ref{thm:p6}. We assume that  $\bal_3(n, P_{3k'})\leq (k'+o(1))n$ for any $2 \le k' < k$. Suppose now we have a complete graph on $n$ vertices whose edges are colored with red, blue and green such that there are at least $(k+\ep)n$ edges from each color. By the induction hypothesis, there is a balanced $P_{3(k-1)}$, say $P$. Let $C = V(P)$ and $D = V(K_n) \setminus C$. We will distinguish two cases.\\

\noindent
\textit{\textbf{Case 1: $D$ has no edges from at least one color.}} \\
 Assume, without loss of generality, that red is missing in $D$. Then $e_r(C,D) = e_r(K_n) - e_r(C) \ge kn$, if $n$ is large enough.\\
 
 \noindent
 \textit{\textbf{Subcase 1.1: Suppose that $e_b(D) \ge \left(\frac{k-1}{2}+\epsilon \right) n$ and $e_g(D) \ge \left(\frac{k-1}{2}+\epsilon \right) n$.}} \\
 Then, by Theorem \ref{thm:bal_2_paths}, there is a balanced blue-green $2k$-path, say $Q$, contained in $D$. Moreover, since $n$ is large, we still have $e_r(C, D \setminus V(Q)) \ge kn$. Hence, because of (\ref{eq:ErGa}), there is a red $k$-path, say $R$, contained in the red graph induced by the edge set $E_r(C, D \setminus V(Q))$.  Let $a, b \in Q$  be the end-vertices of the path $Q$, and let  
$u, v \in V(R)$ be the end-vertices of the path $R$. 
We will show that there is a way of connecting both paths to have a balanced $3$-colored $P_{3k}$. Consider the cycle $S$ formed by the two paths $Q$ and $R$ and the edges $au$ and $bv$. Let $v'v$ be the last edge in path $R$. If the edge $ua$ is red, we can easily see that $S-v$ is a balanced $3$-colored $3k$-path. Hence, $au$ is either green or blue. The same occurs with the edge $bv$. Suppose now $au$ and $bv$ are both of the same color, say blue. Then the path consisting of the blue-green path $Q$ together the edges $au$ and $bv$ has $k+2$ blue edges and $k$ green edges. Thus, it has to contain two consecutive blue edges, say $xy$ and $yz$. Then $S - y$ is a balanced  $3$-colored $3k$-path. If, on the other side, $au$ and $bv$ are one blue and one green, then we can take two consecutive edges $xy$ and $yz$ from $Q$ such that they have different color and we can see that $S - y$ is a balanced $3$-colored $3k$-path. \\

\noindent
\textit{\textbf{Subcase 1.2: Suppose $e_b(D) < \left(\frac{k-1}{2}+\epsilon \right) n$ or $e_g(D) < \left(\frac{k-1}{2}+\epsilon \right) n$.}} \\
 Without loss of generality, we assume $e_b(D) < \left(\frac{k-1}{2}+\epsilon \right) n$. Then, since for $n$ large enough we have $e_b(K_n) - e_b(C)  \ge kn$, and it follows that
\[e_b(C,D) = e_b(K_n) - e_b(C) - e_b(D) \ge kn - \left(\frac{k-1}{2}+\epsilon \right) n =  \left(\frac{k+1}{2}-\epsilon \right)n.\]
We will show first that there is a blue, a red, and a green $k$-path, pairwise disjoint. By the above inequality, $e_b(C,D) \ge \left(\frac{k+1}{2}-\epsilon \right)n.$ and so by (\ref{eq:ErGa}), the graph induced by the edges contained in $E_b(C,D)$ contains a path of length $k+1$, say $B'$. If $k$ is odd, let $x$ be one of the two end-vertices  of $B'$. If $k$ is even, $B'$ has one end-vertex in $C$ and one in $D$. In this case, let $x$ be the end-vertex of $B'$ contained in $C$. Now define $B = B' - x$. Observe that, for $k$ even or odd, we have 
\begin{equation}\label{eq:C_cap_B}
|C \cap V(B)| = \left \lfloor\frac{k+1}{2}\right\rfloor.
\end{equation} 
We will show now that $e_r(C \setminus V(B), D) \ge \frac{k-1}{2}n$ by using the fact that
 \[e_r(C \cap V(B) , D) \le |C \cap V(B)| |D| \le \left \lfloor\frac{k+1}{2}\right\rfloor n,\]
 that holds because of (\ref{eq:C_cap_B}). 
Indeed,
\[e_r(C \setminus V(B), D) = e_r(C,D) - e_r(C \cap V(B), D) \ge kn - \left \lfloor\frac{k+1}{2}\right\rfloor n \ge \frac{k-1}{2}n.\]
 Hence, by Erd\H{o}s-Gallai (\ref{eq:ErGa}), there is a red path on $k$ edges, say $R$, contained in $E_r(C \setminus V(B), D)$.
Finally, consider the set $D' =  D \setminus (V(B) \cup V(R))$. Since there are only green and blue edges in $D$ and $e_b(D) < \left(\frac{k-1}{2} - \epsilon \right) n$ by hypothesis of this case, we have
\[e_g(D') = \binom{|D'|}{2} - e_b(D') \ge \binom{|D'|}{2} - \frac{k-1}{2}n = \Theta(n^2).\]
Hence, the set $D'$ has $\Theta(n^2)$ green edges and thus clearly it has to contain a green $k$-path $G$, again by (\ref{eq:ErGa}).\\

We will show now that with these three paths $B$, $R$ and $G$, we can construct a balanced $3$-colored $P_{3k}$. Let $b, b'$ be the end-vertices of $B$, $r,r'$ the end-vertices of $R$ and $g,g'$ the end-vertices of $G$ and consider the cycle $S$ formed by the three paths together with the edges $b'r$, $r'g$ and $g'b$. Suppose two of the edges $b'r$, $r'g$, $g'b$ are of the same color, say green. We call $c$ the color of the third edge. Then, no matter where this $c$-colored edge is situated, we can easily see that there is a $3$-path $wxyz$ contained in the cycle such that $wx$ and $xy$ are green and $yz$ has color $c$. Deleting the vertices $x$ and $y$ from $S$ gives a balanced $P_{3k}$ and we are done. Hence we may assume that the edges $b'r$, $r'g$ and $g'b$ are all from different color. If $b'r$ is green, we can consider the $3k$-path that is formed by deleting vertices $b'$ and $r$ from $S$. Similarly happens if $r'g$ is blue or $g'b$ is red. Hence, we may assume that the cycle $S$ consists of a blue, a green and a red $P_{k+1}$ glued together. Without loss of generality, we assume that $b'r$ is red, $r'g$ is green and $g'b$ is blue. Let $r''$ be the neighbor of $r'$ on the red path $R$ and consider the edge $gr''$ and the cycle $S' = (S - r') + gr''$. If $gr''$ is blue, then $S' - b$ is a balanced $P_{3k}$. If $gr''$ is red, then $S' - b'$  is a balanced $P_{3k}$. Finally, if $gr''$ is green, then $S' - g'$ is a balanced $3k$-path.\\

\noindent
\textit{\textbf{Case 2: $D$ has edges from all three colors.}}\\
If there is a rainbow triangle in $D$, we are done by Lemma \ref{lem:norainbow}. If there is no rainbow triangle in $D$, then the $3$-coloring in $D$  is a Gallai coloring (see \cite{Gallai, GySi, KoSiTu}) and, by \cite{BiDiVo} (see also \cite{GSSS}), we know that the graph induced by the edges of one of the colors is spanning in $D$. Without loss of generality, assume that the spanning color in $D$ is green. Let $x$ and $y$ be the end-vertices of the path $P$. 

From here, we build the proof by contradiction, assuming that there is no $3$-colored balanced $P_{3k}$. \\

\noindent
\textit{\textbf{Claim 1:}} There is no red-blue $P_2$ in $D$. \\
If $abc$ is a red-blue $2$-path in $D$ we can  construct a balanced $P_{3k}$ as follows. Since there are no rainbow triangles, assume without lost of generality that $ac$ is red. Since color green is spanning in $D$, we know there are vertices  $d,d'\in D$ such that $da$ and $d'b$ are green (where $d=d'$ is possible).  Observe now that, depending on the color of $cx$, one of $abcP$, $d'bcP$ or $dacP$ will be a balanced $P_{3k}$.  \hspace*{\fill}$\diamond$\\

\noindent
\textit{\textbf{Claim 2:} All edges from $\{x,y\}$ to $D$ are green.}\\
Suppose that there is a vertex $c\in D$ such that $xc$ is not green, say, without lost of generality, it is blue. If there is a vertex $b\in D$ such that $bc$ is red, then we can easily build a  balanced $3k$-path, namely $abcP$, where $a$ is such that $ab$ is green (recall that color green is spanning in $D$). On the other hand, if there is no red edge in $D$ incident to $c$, consider a red edge $ab$ in $D\setminus\{c\}$. By Claim 1, $ac$ is not blue, so it has to be green and, again, we can easily build a  balanced $P_{3k}$. Hence, all edges from $x$ to $D$ are green. By symmetry, all edges from $y$ to $D$ are green, too.
 \hspace*{\fill}$\diamond$\\

\noindent
For the next claims we consider  a balanced $3$-path $Q=abcd$ in $D$, which exists because of Lemma \ref{prop:p3is0}. By Claim 1, the middle edge is green and, moreover, all remaining edges induced by $V(Q)$ are green.\\

\noindent
\textit{\textbf{Claim 3:}} The end-edges of $P$ are not green.\\
Suppose one end-edge of $P$ is green, say the edge incident to $y$. Then the path $abcdP-y$ is a  balanced $3k$-path  since, according to Claim 2, $dx$ is green. \hspace*{\fill}$\diamond$\\

\noindent
\textit{\textbf{Claim 4:}} There are no consecutive green edges in $P$.\\
If $uvw$ is a green-green path in $P$, consider the cycle $S=abcPa$.  Then, $S-v$ is a   balanced $P_{3k}$. \hspace*{\fill}$\diamond$\\

\noindent
\textit{\textbf{Claim 5:}} The edge $xy$ is green.\\
Suppose that $xy$ is not green, say, without lost of generality, it is red.  Consider the cycle $S=Px$ that has $k-1$ blue edges, $k-1$ green edges and $k$ red edges. Note that, for any red edge $uv$ in $P$, the path obtained from $S$ by removing the edge $uv$ plays the same role as $P$. Thus, by Claim 3, the end-edges of this new path, $S-\{uv\}$, are not green. This means that no red edge  can be next to a green edge in $S$. Then, since green edges form a matching, each green edge has to be preceded and succeeded by a blue edge, implying that there have to be more blue edges than green edges in $S$, which is a contradiction.
 \hspace*{\fill}$\diamond$\\

Now that we know that the edge $xy$ is green we are going to work with the cycle $S=Px$, which  has $k-1$ blue edges, $k-1$ red edges and $k$ green edges, where the green edges form a matching. Recall that $V(P)=C=V(S)$ \\

\noindent
\textit{\textbf{Claim 6:}} There are at most $k-2$ vertices in $C$ incident on red or blue edges from $E(C,D)$. \\ 
Let $uv$ be a green edge in $S$ and consider  the path $S-\{uv\}$. Note that this new path, $S-\{uv\}$, plays the same role as $P$. Then, by Claim 2, all edges from  $\{u,v\}$ to $D$ are green. Since there are no adjacent green edges in $S$,  this leaves us $(3k-2)-2k=k-2$  vertices allowed to send red or blue edges to $D$.  \hspace*{\fill}$\diamond$\\

By a simple counting argument, it must be that $S$ contains a subpath $uvwz$ which is green-notgreen-green. Suppose, without lost of generality, that $vw$ is red.\\

\noindent
\textit{\textbf{Claim 7:}} The graph induced by the red edges in $D$ is a matching.\\
If this is not the case, take a red-red path, $abc$, in $D$. Consider a blue edge $de$ in $D$. By Claim 1, we know that $\{d,e\}\cap\{a,b,c\}=\emptyset$ and the edge $cd$ is green. Let $P'=S-\{v,w\}$. Note that $P'$ is a $(3k-5)$-path with  $k-1$ blue edges, $k-2$ red edges and $k-2$ green edges. Hence, $abcdeP'$ is a  balanced $P_{3k}$, as the edges from $u$ (or $z$) to $D$ are green, by considering $S\setminus\{wz\}$ (or $S\setminus\{uv\}$). \hspace*{\fill}$\diamond$\\

Now by using Claims 6 and 7, we can count the maximum number of red edges in order to get a contradiction. 
\begin{align*}
e_r(K_n) &=e_r(C,D)+ e_r(D)+   e_r(C)\\
&\le  (k-2)|D|+\frac{|D|}{2}+o(n) \\
&\le \left( k - \frac{3}{2}\right) n + o(n) \\
&=  \left(\left( k - \frac{3}{2}\right) + o(1) \right)n,
\end{align*}
which is not possible by hypothesis.
Hence, we have shown that, for any $0 < \epsilon < \frac{1}{2}$, $\bal_3(n, P_{3k}) < (k+\epsilon)n$ and we obtain the desired bound (\ref{eq:ub}).
\end{proof}
\section{Balanced paths with many colors}
In this section our goal is to prove Theorems \ref{negative} and \ref{positive}. The former shows that it is hopeless to find a straightforward generalization of our bounds on $\bal_r(n, P_{rk})$ when $r=3$ to arbitrary $r$, as it will be the case that for any odd $k$, there are infinitely many $r$ such that $\bal_r(n, P_{rk})=\Omega(n^2)$. On the other hand, the latter will give some hope of generalizing our bounds under the assumption that $k$ is even and sufficiently large with respect to $r$. 
\begin{proof}[Proof of Theorem \ref{negative}]
Given an odd $k$, we consider an $r$ of the form $r=\binom{l}{2}+1$, with $l\geq 4$, and $l$ even. As there are infinitely many such $r$, it will be enough to display an $r$-colored complete graph on $n$ vertices with $\Theta(n^2)$ edges in each color that avoids a balanced embedding of a $P_{rk}$.
\par We divide the vertices of $K_n$ into $l$ sets of size as equal as possible, and call them $V_i$, $i\in[l]$. We color each complete bipartite graph $V_i\times V_j$, $i\neq j$, a different color. With the remaining color, say $c_l$, we color everything else. That is, we color all edges contained in any of the $V_i$'s the same color. 
\par Towards a contradiction, assume that we have a balanced embedding of $P_{rk}$ into this coloring. We will now define an auxiliary multi-graph with vertex set $\{V_i\}_{i\in [l]}$. If we contract all edges of color $c_l$ of the balanced $P_{rk}$ in the embedding, we obtain a balanced copy of $P_{(r-1)k}$. (This copy will not necessarily have a valid embedding onto the original graph, but we will simply use this $P_{(r-1)k}$ to define the edges of the auxiliary multigraph.) 
\par Every time the balanced $P_{(r-1)k}$ uses an edge between $V_i$ and $V_j$, we add a new edge between the corresponding vertices in the multi-graph. 
Observe now that the embedding in the original graph corresponds to an Eulerian trail in the auxiliary multigraph. Thus each vertex in the multigraph corresponding to a $V_i$ which does not contain the start or end vertex of the embedding must have even degree. Observe further that such a vertex in the multigraph must exist, by our assumption that $l\geq 4$. Finally, notice that such a vertex sends an odd number of edges to each of its neighbors, as $k$ is odd, and it has an odd number of neighbors, as $l$ is even, contradicting the fact that this vertex has even degree. 
\end{proof}
\begin{proof}[Proof of Theorem \ref{positive}]
Given $r$, our goal this time is to find a balanced copy of $P_{2rk}$ (for $k$ sufficiently large) in a coloring of a $K_n$ for $n$ sufficiently large. Let $k_0$ be sufficiently large, it's value to be specified later in the proof. Let $k\geq k_0$, and we choose $n$ so that an $r$-coloring of a $K_n$ with $C_{r,k}n^{2-1/10rk}=o(n^2)$ edges in each color class contains a member from $ \mathcal{F}_{10rk}^r$, by Theorem \ref{thm:extremalmulticolorbollobas}. 
\par Fix such a coloring of a $K_n$, and call the resulting unavoidable graph $K$. This $K$ will require us to select a particular value of $k_0'$, for which $k\geq k_0'$ will be sufficient for the argument to go through. As it will follow that we get a corresponding value of $k_0'$ for every possible $K \in \mathcal{F}_{10rk}^r$ we end up with, we will set $k_0$ in the beginning of the proof as the maximum over all such $k_0'$ (recall that there are only finitely many unavoidable configurations for each $r$). Hence for $k\geq k_0$, the embedding argument will go through, regardless of which graph $K \in \mathcal{F}_{10rk}^r$ we start with.   
\par Now, having fixed a $K$, we will specify an embedding of $P_{rk}$ into $K$, and observe that we will never run out of edges in the process, as we were generous in the beginning, and each part in $K$ has more than enough vertices to complete the embedding. 
\par Let $l$ be the number of parts in $K$ (note $l\leq 2r-2$ by the remark in the Introduction), let $V_i$ ($i\in [l]$) be the parts in $K$, and let $c_i$ ($i\in[r]$) be the set of colors. Define $$\mathcal{C}:=\{c_i\colon c_i \text{ is the color of one of the complete bipartite graphs } V_i\times V_j \text{ with } i\neq j \}.$$
\par So $\mathcal{C}$ is the set of all colors that appear on edges going between the parts of $K$. 
\begin{claim} It is sufficient to embed a balanced $P_{|\mathcal{C}|k}$ into $K\setminus \bigcup_{i\in[l]}E(V_i)$ \end{claim}
 As $K \in \mathcal{F}_{10rk}^r$, it uses all $r$ colors. If $|C|=r$, the claim is obvious. Otherwise, there are colors which appear only inside one of the $V_i$. For any such color, note it is easy to insert a path of length $k$ from that color into whatever embedding we already have, using only edges from that clique. This simple observation establishes the claim. 
 \par We proceed to find an embedding of $P_{|\mathcal{C}|k}$ in $K$ that uses only edges from the complete bipartite graphs. We do this again by associating an auxiliary graph, and arguing this auxiliary graph contains a Eulerian circuit. As we will find a circuit, our result will be valid even if we were trying to find a balanced cycle, as remarked in the Introduction.
 \par For $c_i\in \mathcal{C}$, let $\#c_i$ be the number of times the color $c_i$ appears in one of the bipartite graphs $V_i\times V_j$. We now at last specify our choice of $k'_0$:
 $$k'_0 := \text{lcm}({\#c_i \colon c_i\in\mathcal{C}})$$
 where lcm denotes the least common multiple of a set. For $k\geq k'_0$, we will now find a $P_{|\mathcal{C}|k}$ in $K$ minus the edges in one of the cliques.
 \par As in the previous proof, we define a multi-graph for which the vertices correspond to the $V_i$, and put $2k'_0/\#c_i$ many edges between each vertex corresponding to $V_i$ and $V_j$, where $c_i$ is of course the color of $V_i\times V_j$ in $K$. As $2k'_0/\#c_i$ is an even integer, we find a Eulerian circuit in our graph, which naturally induces an embedding of a balanced $P_{2kr}$ in $K$, allowing us to conclude.
\end{proof}

\section{Discussion}
\par We conjecture that removing the $r^r$ term in the exponent in Theorem \ref{thm:extremalmulticolorbollobas} is possible, and this remains an intriguing problem (Conjecture~\ref{mainconjecture}). We have already expanded on some tools which we believe will be useful in Section \ref{sec:improvements}. Such a bound would be essentially tight, as further improvements would have to improve the best known bounds on $\textrm{ex}(n, K_{t,t})$.
\par Closing the gap from Theorem \ref{thm:generalbounds} is also an open problem. We conjecture that the lower bound should be the truth, but new ideas are needed to improve on our upper bound.
\par A perhaps more elementary question that remains is determining for which graphs $G$ we have $\bal_r(n,G)<\infty$. We call an embedding of $G$ into an $r$-edge-colored $K_n$ a \textit{balanced embedding} if all $r$ color classes are almost equally represented as in Definition \ref{def:bal_r}. We have the following abstract characterization.
\begin{proposition}
Let $G$ be any graph of order $q$, and $r\geq 2$ an integer. Then $\bal_r(n, G)=o(n^2)$ if and only if there exists a balanced embedding of $G$ into all patterns from $\mathcal{F}_q^r$.
\end{proposition}
\begin{proof}
If $G$ admits a balanceable embedding into all patterns from $\mathcal{F}_q^r$, by Theorem \ref{thm:extremalmulticolorbollobas} we will be able to find balanced copies of $G$ in sufficiently large complete graphs with $\Theta(n^{2-1/v(G)})$ edges in each color. On the other hand, if $G$ does not admit a balanced embedding into a particular pattern from $\mathcal{F}_q^r$, we can use this pattern to construct a $K_n$ with $\Theta(n^2)$ edges in each color class, hence it cannot be that $\bal_r(n, G)=o(n^2)$.
\end{proof}
\par This characterization does not rule out the possibility that there might be balanceable graphs with $\bal_r(n, G)=\Theta(n^2)$. We find this to be unlikely, but we are unable to prove it for $r \ge 3$,  in contrast to the case $r=2$,  where it was shown that balanceable graphs have always $\bal(n, G)=o(n^2)$ \cite{caro}. For example, by investigating Figure \ref{fig:F3}, one can verify that $\bal_3(n, C_{6k})=o(n^2)$, and $\bal_3(n,C_{6k+3})=\Omega(n^2)$. We leave it as an open problem to determine whether $\bal_3(n,C_{6k+3})=\infty$. \par It would in general be interesting to obtain results for cycles for $r \ge 3$ (see \cite{DEHV} for the case $r=2$).

\section*{Acknowledgements}

The present work was concluded during the workshop \emph{Zero-Sum Ramsey
Theory: Graphs, Sequences and More} (19w5132). We thank the facilities provided by the BIRS-CMO research station. 

The second author was partially supported by PAPIIT IN111819 and CONACyT project 282280. The third author was partially supported by PAPIIT IN116519 and CONACyT project 282280.


\end{document}